\UseRawInputEncoding
\documentclass[12pt]{article}
\usepackage[centertags]{amsmath}
\usepackage{amsfonts}
\usepackage{amssymb}
\usepackage{latexsym}
\usepackage{amsthm}
\usepackage{newlfont}
\usepackage{graphicx}
\usepackage{listings}
\usepackage{booktabs}
\usepackage{abstract}
\usepackage{enumerate}
\usepackage{xcolor}
\RequirePackage{srcltx}
\date{}

\newlength{\defbaselineskip}
\newcommand{\setlinespacing}[1]%
           {\setlength{\baselineskip}{#1 \defbaselineskip}}

\newcommand{\actaqed}{\hfill $\actabox$}
{\medskip\noindent \textit{Proof of #1. }}%
{\actaqed \medskip}

\def\cA{{\mathcal A}}

\def\cC{{\mathcal C}}
\def\cD{{\mathcal D}}

\def\cK{{\mathcal K}}
\def\cL{{\mathcal L}}
\def\cM{{\mathcal M}}

\def\cS{{\mathcal S}}
\def\cT{{\mathcal T}}

\def\cV{{\mathcal V}}

\def\bbC{{\mathbb C}}

\def\bbN{{\mathbb N}}

\def\bbR{{\mathbb R}}

\def\bbT{{\mathbb T}}

\def\bbZ{{\mathbb Z}}

\def\bF{{\mathbf F}}

\def\bH{{\mathbf H}}

\def\bK{{\mathbf K}}

\def\bN{{\mathbf N}}

\def\bW{{\mathbf W}}

\def\ba{\mathbf a}
\def\bb{\mathbf b}

\def\bj{\mathbf j}
\def\bk{\mathbf k}

\def\bp{\mathbf p}
\def\bq{\mathbf q}
\def\br{\mathbf r}
\def\bs{\mathbf s}
\def\btt{\mathbf t}

\def\bx{\mathbf x}
\def\by{\mathbf y}
\def\bz{\mathbf z}
 
 \def \<{\langle}
\def\>{\rangle}
\def \La{\Lambda}
\def \Og{\Omega}

\def \e{\varepsilon}

\def \ff{\varphi}
\def\al{\alpha}
\def\bt{\beta}
\def \ro{\varrho}

\def\la{\lambda}

\def \sp{\operatorname{span}}

\def\bt{\beta}

\newtheorem{Theorem}{Theorem}[section]
\newtheorem{Lemma}{Lemma}[section]
\newtheorem{Definition}{Definition}[section]
\newtheorem{Proposition}{Proposition}[section]

\numberwithin{equation}{section}

\newcommand{\be}{\begin{equation}}
\newcommand{\ee}{\end{equation}}

\begin{document}

\title{Sampling recovery on classes defined by integral operators and sparse approximation with adaptive dictionaries}

\author{V. Temlyakov}

\newcommand{\Addresses}{{% additional braces for segregating \footnotesize
  \bigskip
  \footnotesize
 
  \medskip
  V.N. Temlyakov, \textsc{University of South Carolina, USA,\\ Steklov Mathematical Institute of Russian Academy of Sciences, Russia;\\ Lomonosov Moscow State University, Russia; \\ Moscow Center of Fundamental and Applied Mathematics, Russia.\\  
E-mail:} \texttt{temlyakovv@gmail.com}

}}
\maketitle

\begin{abstract}{In this paper we continue to develop the following general approach. We study asymptotic behavior of the errors of sampling recovery not for an individual smoothness class, how it is usually done, but for the collection of classes, which are defined by integral operators with kernels coming from a given class of functions. Earlier, such approach was realized for the Kolmogorov widths and very recently for the entropy numbers. It turns out that the above problem is closely related to the sparse approximation problem with respect to different redundant dictionaries. Specifically, the problem of sampling recovery is connected with sparse nonlinear approximation with respect to adaptive dictionaries, which means that the dictionary depends on the function under approximation.}
\end{abstract}

\section{Introduction}
\label{In}

 This paper is a followup to the very recent paper \cite{VT212}. We continue to develop the following general setting, which was formulated in \cite{VT31}. 
A typical smoothness class can be defined with a help of an integral operator with a special kernel. For instance, in the case 
of periodic functions the Bernoulli kernel can be used. In \cite{VT31} we suggested to study asymptotic characteristics
(Kolmogorov widths) not for an individual smoothness class but for the collection of classes, which are defined by integral operators with kernels coming from a given class of functions. In \cite{VT212} we followed the setting of \cite{VT31} and obtained some results in the case, when the asymptotic characteristic under investigation is the entropy numbers. 
In this paper we continue to follow the above pattern developed in \cite{VT31} and prove the corresponding results 
in the case, when the asymptotic characteristic is the error of sampling recovery. 

We now proceed to the detailed presentation. Let $(\Omega,\mu)$ be a probability space. 
  By the $L_p$, $1\le p< \infty$, norm we understand
$$
\|f\|_p:=\|f\|_{L_p(\Omega,\mu)} := \left(\int_\Omega |f|^p\,d\mu\right)^{1/p}.
$$
By the $L_\infty$-norm we understand the uniform norm of continuous functions
$$
\|f\|_\infty := \sup_{\omega\in\Omega} |f(\omega)|
$$
and with some abuse of notation we occasionally write $L_\infty(\Omega)$ for the space $\cC(\Omega)$ of continuous functions on $\Omega$. We define the vector $L_{\bp}$-norm, $\bp=(p_1,\dots,p_v)$, of functions of $v$ variables $\bx=(x_1,\dots,x_v)$ as
$$
\|f(\bx)\|_{\bp} :=\|f(\bx)\|_{(p_1,\dots,p_v)} :=\|f(\bx)\|_{p_1,\dots,p_v} := \|\cdots\|f(\cdot,x_2,\dots,x_v)\|_{p_1}\cdots\|_{p_v}.
$$

\subsection{Sampling recovery}
\label{Insr}

We begin with a formulation of the general problem, which we study here in a special case of sampling recovery. We give this formulation in the case of $d$-variate classes $\bW^K_q$ (see below) and note that we mostly study the univariate case ($d=1$) in this paper.  We stress that the problem is of high interest and importance for all dimensions $d$. It seems like in the case $d>1$ it is even more difficult than in the case $d=1$, where it is not yet fully solved. 

{\bf Problem $K$. Formulation in general setting.} Let $\Omega^i$, $i=1,2$, be compact sets in $\bbR^d$ with probability measures $\mu_i$ on them.  Let $K(\bx,\by)$ be a measurable with respect to $\mu_1\times\mu_2$ function on $\Omega^1\times\Omega^2\subset \bbR^{2d}$.
Assume that for all $\ff\in L_q(\Omega^2,\mu_2)$   the integral
$$
I_K(\ff):=\int_{\Omega^2}K(\bx,\by)\varphi(\by)d\mu_2
$$
exists for all $\bx\in\Omega^1$. 

 Let $1\le q\le \infty$ and $q':= q/(q-1)$ be the dual to $q$.  Then we define the class
\be\label{In1}
\bW^K_q :=\left\{f:f(\bx) = \int_{\Omega^2}K(\bx,\by)\varphi(\by)d\mu_2,\quad\|\varphi\|_{L_q(\Omega^2,\mu_2)}\le 1\right\}.  
\ee
Clearly, the class $\bW^K_q $ is the image of the unit ball of the space $L_q(\Omega^2,\mu_2)$ of the integral operator $I_K$. 

We are interested in the following problem formulated in \cite{VT212} (see also \cite{VT31}). Assume that a class (a collection) $\bK$ of kernels $K$ is given. Consider a specific asymptotic characteristic $ac_n(\bW^K_q,L_p)$ of classes $\bW^K_q$ in the space $L_p(\Omega^1,\mu_1)$,
$1\le q,p\le \infty$. We want to estimate the following characteristics of the collection $\bK$
$$
ac_n(\bK,L_q,L_p) := \sup_{K\in\bK}  ac_n(\bW^K_q,L_p).
$$

The above problem has two important ingredients -- the collection $\bK$ of kernels of integral operators and the asymptotic characteristic $ac_n(\cdot,\cdot,\cdot)$. The above problem is relatively well studied in the case, when the asymptotic characteristic $ac_n$ is the Kolmogorov width $d_n$. The reader can find the corresponding results in \cite{VT31}, \cite{VT32},  and \cite{VT47}. A brief description of some of those results can be found in \cite{VT211}. 
For illustration we formulate one of those results.    The following result is from \cite{VT32}. We need the following notation for $1\le q,p\le\infty$
 \be\label{ksi'}
 \xi(q,p):= \left(\frac{1}{q} - \max\left(\frac{1}{2},\frac{1}{p}\right)\right)_+, \quad  (a)_+ :=\max(a,0).
\ee

 We also use the following convenient notations. We use $C$, $C'$ and $c$, $c'$ to denote various positive constants. Their arguments indicate the parameters, which they may depend on. Normally, these constants do not depend on a function $f$ and running parameters usually denoted by $m$, $n$, $k$. We use the following symbols for brevity. For two nonnegative sequences $a=\{a_n\}_{n=1}^\infty$ and $b=\{b_n\}_{n=1}^\infty$ the relation $a_n\ll b_n$ means that there is  a number $C(a,b)$ such that for all $n$ we have $a_n\le C(a,b)b_n$. Relation $a_n\gg b_n$ means that 
 $b_n\ll a_n$ and $a_n\asymp b_n$ means that $a_n\ll b_n$ and $a_n \gg b_n$. 
 For a real number $x$ denote $[x]$ the integer part of $x$.
  
 Here is a known result on the Kolmogorov widths.   Let $X$ be a Banach space and $\bF\subset X$ be a  compact subset of $X$. The quantities  
$$
d_n (\bF, X) :=  \inf_{\{u_i\}_{i=1}^n\subset X}
\sup_{f\in \bF}
\inf_{c_i} \left \| f - \sum_{i=1}^{n}
c_i u_i \right\|_X, \quad n = 1, 2, \dots,
$$
are called the {\it Kolmogorov widths} of $\bF$ in $X$.

 \begin{Theorem}[{\cite{VT32}}]\label{InT1} Let $d=1$ and $\bF^\br_1$ denote one of the classes $\bW^\br_1$ or $\bH^\br_1$ (see the definition in Section \ref{fc} below) of functions of two variables. Then for $1\le q,p \le \infty$ and $\br > (1,1+\max(1/2,1/q))$   we have 
 $$
\sup_{K\in \bF^\br_1}  d_n(\bW^K_q)_p \asymp n^{-r_1-r_2 + \xi(q,p)} 
 $$
 with $\xi(q,p)$ defined in (\ref{ksi'}).
 \end{Theorem}

 In this paper we study the case, when the asymptotic characteristic is the error of sampling recovery and the collection $\bK$ is a class of multivariate periodic functions on $2d$ variables with mixed smoothness. Recall the setting 
 of the optimal linear recovery introduced in \cite{VT51}. For a fixed $m$ and a set of points  $\xi:=\{\xi^j\}_{j=1}^m\subset \Omega$, let $\Phi $ be a linear operator from $\bbC^m$ into $L_p(\Omega,\mu)$.
Denote for a class $\bF$ (usually, centrally symmetric and compact subset of $L_p(\Omega,\mu)$)
$$
\varrho_m(\bF,L_p) := \inf_{\xi} \inf_{\text{linear}\, \Phi } \sup_{f\in \bF} \|f-\Phi(f(\xi^1),\dots,f(\xi^m))\|_p.
$$
The above described recovery procedure is a linear procedure. 

We now formulate the main results of this paper. The corresponding proofs are given in Section \ref{Pr}. 

 \begin{Theorem}\label{InT2} Let $d=1$ and $\bF^\br_1$ denote one of the classes $\bW^\br_1$ or $\bH^\br_1$ (see the definition in Section \ref{fc} below) of functions of two variables.  Then for $1\le q \le  2\le p\le \infty$, and $\br > (1,1+ 1/q)$   we have 
 $$
\sup_{K\in \bF^\br_1}  \ro_{m}(\bW^K_q,L_p) \asymp  m^{-r_1-r_2 +1/q-1/p}.
 $$
 \end{Theorem}
 
 We introduce some more notation. 
  For $\bq =(q_1,q_2)$,  $1\le q_1,q_2 \le \infty$,  denote
 $$
 \br(\bq) := \begin{cases}  
 (1/q_1,1/q_2), & 2\le q_1 \le \infty,  \\
 (1/q_1,\max(1/2,1/q_2)), & 1\le q_1<2 .
 \end{cases}
 $$
 
\begin{Theorem} \label{InT3} Let $d=1$ and $\bF^\br_\bq$ denote one of the classes $\bW^\br_\bq$ or $\bH^\br_\bq$ (see the definition in Section \ref{fc} below) of functions of two variables.  Then for   $1\le q_1,q_2 \le \infty$, $2\le p\le \infty$,   and $\br > \br(\bq)$   we have 
 $$
\sup_{K\in \bF^\br_\bq}  \ro_m(\bW^K_1,L_p) \ll m^{-r_1-r_2 + (1/q_1-1/2)_+ +1/2-1/p } .
 $$
 \end{Theorem}
 
 We now give the lower bounds for the following characteristic of the optimal sampling recovery. Let $\Omega$ be a compact subset of $\bbR^d$ with the probability measure $\mu$ on it.
For a function class $W\subset \cC(\Omega)$,  we  define (see \cite{TWW})
$$
\varrho_m^o(W,L_p) := \inf_{\xi } \inf_{\cM} \sup_{f\in W}\|f-\cM(f(\xi^1),\dots,f(\xi^m))\|_p,
$$
where $\cM$ ranges over all  mappings $\cM : \bbC^m \to   L_p(\Omega,\mu)$  and
$\xi$ ranges over all subsets $\{\xi^1,\cdots,\xi^m\}$ of $m$ points in $\Og$.
Here, we use the index {\it o} to mean optimality. 
Clearly,
\be\label{In2}
\varrho_m^o(W,L_p) \le \varrho_m(W,L_p)
\ee
and, therefore, the lower bounds for the $\varrho_m^o$ serve as the lower bounds for the $\varrho_m$ and the upper bounds for the $\varrho_m$ serve as the upper bounds for the $\varrho_m^o$.
 
\begin{Theorem} \label{LbT1} Let $d=1$ and $\bF^\br_\bq$ denote one of the classes $\bW^\br_\bq$ or $\bH^\br_\bq$ (see the definition in Section \ref{fc}) of functions of two variables.  Then for \newline  $1\le q_1,q_2 \le \infty$, $1\le q\le p\le \infty$,   and $\br > (1/q_1,1/q_2)$   we have 
 $$
\sup_{K\in \bF^\br_\bq}  \ro_m^o(\bW^K_q,L_p) \gg m^{-r_1-r_2 -1 +1/q_1 +1/q -1/p } .
 $$
 \end{Theorem}
 Note that in the case $q_1=1$ Theorem \ref{LbT1} provides the required lower bounds in Theorem \ref{InT2}. 
In the particular case $q=1$ Theorem \ref{LbT1} gives
$$
\sup_{K\in \bF^\br_\bq}  \ro_m^o(\bW^K_1,L_p) \gg m^{-r_1-r_2  +1/q_1  -1/p } .
 $$
 This lower bound coincides with the upper bound in Theorem \ref{InT3} in the case $1\le q_1\le 2$. Therefore, 
 Theorem \ref{InT3} is sharp in the case $1\le q_1\le 2$. We formulate this as a separate statement. 
 
 \begin{Theorem} \label{LbT2} Let $d=1$ and $\bF^\br_\bq$ denote one of the classes $\bW^\br_\bq$ or $\bH^\br_\bq$ (see the definition in Section \ref{fc}) of functions of two variables.  Then for \newline  $1\le q_2\le\infty$, $1\le q_1\le2\le p\le \infty$,   and $\br > \br(\bq)$   we have 
 $$
  \sup_{K\in \bF^\br_\bq}  \ro_m^*(\bW^K_1,L_p) \asymp m^{-r_1-r_2 + 1/q_1 -1/p }, 
 $$
 where $\ro_m^*$ stands for both $\ro_m$ and $\ro_m^o$.
 \end{Theorem}

\subsection{Adaptive dictionaries}
\label{AD}

It turns out that Problem $K$ is closely related to the sparse approximation problem with respect to different redundant dictionaries. We now introduce some concepts from nonlinear sparse approximation and explain the connection. 

The first example of sparse approximation with respect to redundant dictionaries 
 was considered by E. Schmidt in \cite{Sch}, who studied the
approximation of functions $f(x,y)$ of two variables by bilinear forms,
$$
\sum_{i=1}^mu_i(x)v_i(y),
$$ 
in $L_2([0,1]^2)$. In this case we use the following dictionary (bilinear dictionary)
\be\label{Bi1}
\Pi :=  \{u(x)v(y)\,:\, u,v\in L_2([0,1])\},
\ee
where the   functions $u$ and $v$ are functions of a single variable. 
This problem is closely connected with properties of the integral operator
$$
(I_fg)(x) := \int_0^1 f(x,y)g(y) dy
$$
 with kernel $f(x,y)$. The reader can find a detailed discussion of this connection in \cite{VTbook}, Ch.2.

In a general setting we are working in a Banach space $X$ with a redundant system of elements $\cD$ (dictionary $\cD$).   
An element (function, signal) $f\in X$ is said to be $m$-sparse with respect to $\cD$ if
it has a representation $h=\sum_{i=1}^mc_ig_i$,   $g_i\in \cD$, $i=1,\dots,m$, where $\{c_i\}$ are real or complex numbers. The set of all $m$-sparse elements is denoted by $\Sigma_m(\cD)$. For a given element $f$ we introduce the error of best $m$-term approximation
$$
\sigma_m(f,\cD) := \inf_{h\in\Sigma_m(\cD)} \|f-h\|.
$$
We now make a comment on terminology. In the greedy approximation literature we define a dictionary $\cD$ as a system $\{g\}$  of elements $g\in X$ with the following two properties
$$
\|g\|\le 1 \quad\text{for all} \quad g\in X \quad \text{and the closure of}\, \sp(\cD) =X.
$$
The normalization condition $\|g\|\le 1$ is imposed for convenience. Clearly, the characteristic $\sigma_m(f,\cD)$ does not depend on normalisation. In this paper we mostly use this characteristic. 
Let us discuss the second condition. Suppose that a system $\cS\subset X$ does not satisfy this condition. Then, instead of the Banach space $X$ we consider a subspace $X_\cS$ of $X$, which is the closure (in $X$) of $\sp(\cS)$. This makes the system $\cS$ to be a dictionary in the Banach space $X_\cS$. For this reason, we sometimes with a little abuse of exactness freely use both terms {\it system} and {\it dictionary} for a general system. In the greedy approximation theory there are theorems, which guarantee convergence of certain greedy algorithms with respect to any dictionary $\cD$ for any element $f\in X$. Clearly, in the case, when we deal with a system, we can only apply those theorem to $f\in X_\cS$.   

We stress that the bilinear dictionary $\Pi$ does not depend on a function under approximation. In this sense it is not adaptive -- we use it for approximation of all functions. It turns out that in some versions of the Problem $K$ we need to study approximation of a given kernel $K(\bx,\by)$ with respect to a dictionary, which is determined by $K$. We now give a more general definition of the bilinear dictionary  (system) and define three adaptive systems. Let $\bp=(p_1,p_2)$, $1\le p_1,p_2 \le \infty$ be given.

{\bf Bilinear dictionary $\Pi(\bp)$.}   Define
$$
\Pi(\bp) :=\{g: g(\bx,\by) = u(\bx)v(\by),\, u\in L_{p_1}(\Omega^1),\, v\in L_{p_2}(\Omega^2) \}.
$$

 $\cL\cK(\bp)$-{\bf system.} Assume that $K \in L_\bp(\Omega^1\times \Omega^2)$ satisfies the following property. 
For any $\bz \in \Omega^1$ we have $K(\bz,\cdot) \in L_{p_2}(\Omega^2)$. Define
$$
\cL\cK(\bp) :=\{g: g(\bx,\by) = u(\bx,\bz)K(\bz,\by),\,  \forall \bz \,\text{we have}\, u(\cdot,\bz) \in L_{p_1}(\Omega^1)  \}.
$$

 $\cK\cL(\bp)$-{\bf system.} Assume that $K \in L_\bp(\Omega^1\times \Omega^2)$ satisfies the following property. 
For any $\bz \in \Omega^2$ we have $K(\cdot,\bz) \in L_{p_1}(\Omega^1)$. Define
$$
\cK\cL(\bp) :=\{g: g(\bx,\by) = K(\bx,\bz)v(\bz,\by),\,  \forall \bz \,\text{we have}\, v(\bz,\cdot) \in L_{p_1}(\Omega^1)  \}.
$$

$\cK\cK(\bp)$-{\bf system.} Assume that $K \in L_\bp(\Omega^1\times \Omega^2)$ satisfies the following property. 
For any $\ba \in \Omega^1$ we have $K(\ba,\cdot) \in L_{p_2}(\Omega^2)$ and for any $\bb \in \Omega^2$ we have $K(\cdot,\bb) \in L_{p_1}(\Omega^1)$. Define
$$
\cK\cK(\bp) :=\{g: g(\bx,\by) = K(\bx,\bb)K(\ba,\by), \quad (\ba,\bb)\in  \Omega^1\times \Omega^2 \}.
$$
Note that in the literature (see \cite{BS}) the functions $K_{ab}(x,y) := K(x,b)K(a,y)$, $(x,y),(a,b) \in [0,1]^2$ are called {\it cross-functions} of function $K(x,y)$ and the system $\cK\cK(\infty)$ is called the {\it system of cross-functions}. 

It is known (see \cite{VT32}) that the problem of estimating the Kolmogorov widths of classes $\bW^K_q$ (in the case $d=1$), which can be seen as a linear problem, is closely connected with the nonlinear problem of best sparse approximation of the kernel $K$ with respect to the bilinear dictionary $\Pi$. Note that the bilinear dictionary $\Pi$ does not depend on the kernel $K$. 
 
 In this paper (see Section \ref{Pr}) we demonstrate that the problem of estimating the optimal errors of linear recovery of classes $\bW^K_q$ is connected with the nonlinear problem of best sparse approximation of the kernel $K$ with respect to the system (dictionary) $\cL\cK$. Here, the system $\cL\cK$ is determined by the kernel $K$. Thus, we call this type of approximation -- {\it sparse approximation with respect to adaptive systems (dictionaries)}.  
 
 It turns out that the problem of sparse approximation with respect to adaptive systems is not new and has an interesting history, which we now briefly discuss (for more details see \cite{BS} and for another example see Section \ref{D}). It concerns approximation of a function $K(x,y)$ of two variables by linear combinations of its {\it cross-functions} -- sparse approximation with respect to the system $\cK\cK$. Let $K(x,y)$ be a continuous function on $[0,1]^2$. For a point 
 $(a,b)\in [0,1]^2$ we define $K_{ab}(x,y) := K(a,y)K(x,b)$ and call it a cross-function of function $K(x,y)$. 
 
 In 1936 Mazur formulated the following problem (see \cite{[1]}, Problem 153). (M) Is it true that every continuous on $[0,1]^2$ function $K$ can be arbitrarily well approximated in the uniform norm by linear combinations of its cross-functions $\{K_{ab}, (a,b)\in [0,1]^2\}$?
 
 In 1955  Grothendieck (see \cite{[2]}) showed that the Mazur's problem (M) is related to other fundamental problems. Namely, he proved equivalence of the following three statements. 
 
 {\bf G1.} Each Banach space $E$ has the approximation property, which means that for any compact set $S\subset E$ and any $\epsilon >0$ there exists a finite dimensional operator $T$ such that $\|Tx-x\|<\epsilon$ for all $x\in S$. 
 
 {\bf G2.} For any infinite matrix $U:= \left[u_{kj}\right]_{k,j=1}^\infty$ with the properties 
 $$
 \forall{k}\quad \lim_{j\to\infty} u_{kj} =0; \qquad \sum_{k=1}^\infty \max_j |u_{kj}| <\infty,\qquad U^2=0
 $$
 we have
 $$
 \text{tr}(U) := \sum_{k=1}^\infty u_{kk} =0.
 $$
 
 {\bf G3.} For any continuous on $[0,1]^2$ function $K$ the equalities 
 $$
 \int_0^1 K_{ab}(t,t)dt =0, \quad \text{for all} \quad (a,b) \in [0,1]^2
 $$
 imply that
 $$
  \int_0^1 K(t,t)dt =0.
  $$
  
  In 1973 Enflo (see \cite{[3]}) gave an example of a Banach space without the approximation property and, therefore, by the Grothendieck's result on equivalence of {\bf G1} and {\bf G3} the answer to the Mazur's problem (M) is negative. We refer the reader to the very recent paper \cite{BS} for results on approximation of functions by linear combinations of their cross-functions. 
  
In this paper we prove the following corollary of Theorem \ref{InT3}.

\begin{Theorem} \label{InT4} Let $d=1$ and $\bF^\br_\bq$ denote one of the classes $\bW^\br_\bq$ or $\bH^\br_\bq$ (see the definition in Section \ref{fc} below) of functions of two variables.  Then for   $1\le q_1,q_2 \le \infty$, $2\le p\le \infty$,   and $\br > \br(\bq)$   we have 
 $$
\sup_{K\in \bF^\br_\bq}  \sigma_m(K,\cL\cK(\infty))_\infty \ll m^{-r_1-r_2 + (1/q_1-1/2)_+ +1/2 } .
 $$
 \end{Theorem}
 
 Also, we prove the following lower bounds. 
 
 \begin{Theorem} \label{InT4a} Let $d=1$ and $\bF^\br_\bq$ denote one of the classes $\bW^\br_\bq$ or $\bH^\br_\bq$ (see the definition in Section \ref{fc} below) of functions of two variables.  Then for   $1\le q_1,q_2 \le \infty$, $1\le p_1, p_2\le \infty$, and $\br > \br(\bq)$    we have 
 $$
\sup_{K\in \bF^\br_\bq}  \sigma_m(K,\cL\cK(\bp))_\bp \gg m^{-r_1-r_2 + 1/q_1-1/p_1-1/p_2}.
 $$
 \end{Theorem}
 
 Theorems \ref{InT4} and \ref{InT4a} imply the following right order result in the case $\bp=\infty$ and $1\le q_1\le 2$.
 
 \begin{Theorem} \label{InT4b} Let $d=1$ and $\bF^\br_\bq$ denote one of the classes $\bW^\br_\bq$ or $\bH^\br_\bq$ (see the definition in Section \ref{fc} below) of functions of two variables.  Then for     $1\le q_1\le 2$, $1\le q_2 \le \infty$,  and $\br > \br(\bq)$   we have 
 $$
\sup_{K\in \bF^\br_\bq}  \sigma_m(K,\cL\cK(\infty))_\infty \asymp m^{-r_1-r_2 + 1/q_1} .
 $$
 \end{Theorem}

The case of bilinear approximation (non-adaptive case) is studied much better than the above case of the system $\cL\cK$ (adaptive case). The following result is known.

 \begin{Theorem}[{\cite{VT32}}, Theorem 2]\label{InT5} Let $d=1$ and $\bF^\br_\bq$ denote one of the classes $\bW^\br_\bq$ or $\bH^\br_\bq$. Then for $\br > \mathbf{1}$ and $1\le q_1\le p_1 \le \infty$, $1\le q_2,p_2 \le \infty$ we have 
 $$
 \sup_{K\in \bF^\br_\bq} \sigma_m(K,\Pi(\bp))_\bp \asymp m^{-r_1-r_2 + \xi(q_1,p_1)}  
 $$
 where $\xi(q,p)$ is defined in (\ref{ksi'}). 
 \end{Theorem}

In the case $\bp=\infty$, $1\le q_1\le 2$, Theorem \ref{InT5} gives the rate $m^{-r_1-r_2 + 1/q_1-1/2  }$, which is better than in Theorem \ref{InT4b}. It is not surprising because the system $\cL\cK(\bp)$ is a subsystem of the bilinear system $\Pi(\bp)$. 

{\bf Novelty.} Approximation theory plays important role in numerical analysis and in applied mathematics. Driven by applications in big data analysis, image/signal processing, machine learning, and artificial intelligence, approximation theory has moved to the level, where  we try to solve problems in as general formulation as possible. We have moved from univariate functions to multivariate functions, from finite dimensional spaces to infinitely dimensional ones, from classical systems like the trigonometric and algebraic systems to bases in Banach spaces and even father to arbitrary redundant dictionaries in Banach spaces. This paper is in that spirit. We study asymptotic behavior of the errors of sampling recovery not for an individual smoothness class, how it is usually done, but for the collection of classes, which are defined by integral operators with kernels coming from a given class of functions (see Problem $K$ above). Clearly, it is a much more general setting than the one studied before. On this way we have discovered in this paper that in the case of sampling recovery Problem $K$ is closely connected with a new problem of sparse nonlinear approximation with respect to redundant dictionaries, namely, with sparse approximation with adaptive dictionaries (see Subsection \ref{AD} above). We have obtained new results both in sampling recovery on classes defined by integral operators with kernels coming from a given class of functions and in sparse approximation with adaptive dictionaries. Also, another goal of the paper is to attract attention of researchers to the two important new directions mentioned above.

 \section{Function classes}
\label{fc}

As we pointed out in Section \ref{In}  (it is clear from Theorems \ref{InT1} -- \ref{InT5})  we study the case $d=1$, when the kernel $K(x,y)$ is a function of two variables. However, we give the definitions of the classes in the case of any $d\ge1$, because of its importance for all dimensions $d$. In order to avoid a confusion in the future (see Section \ref{kr} below) we use the notation $v$  here. In Theorems \ref{InT1} -- \ref{InT5} we have $v=2d=2$.

We begin with the definition of classes $\bW^\ba_\bq$ (see, for instance, \cite{VTmon}, p.31, in the case of scalar $q$).
\begin{Definition}\label{fcD1}
In the univariate case, for $a>0$, let
\be\label{Bi8}
F_a(x):= 1+2\sum_{k=1}^\infty k^{-a}\cos (kx-a\pi/2)
\ee
be the Bernoulli kernel and in the multivariate case, for $\ba=(a_1,\dots,a_v) \in \bbR^v_+$, $\bx=(x_1,\dots,x_v)\in \bbT^v$, let
\be\label{Bi8m}
F_\ba(\bx) := \prod_{j=1}^v F_{a_j}(x_j).
\ee
Denote for $\mathbf{1}\le \bq\le \infty$ (we understand the vector inequality coordinate wise)
$$
\bW^\ba_\bq := \{f:f=\varphi\ast F_\ba,\quad \|\varphi\|_\bq \le 1\},
$$
where
$$
( F_\ba \ast \varphi)(\bx):= (2\pi)^{-v}\int_{\bbT^v} F_\ba(\bx-\by) \ff(\by)d\by.
$$
\end{Definition}
The classes $\bW^\ba_\bq$ are classical classes of functions with {\it dominated mixed derivative}
(Sobolev-type classes of functions with mixed smoothness).
 
We now proceed to the definition of the classes $\bH^\ba_\bq := \bH^{\ba,v}_\bq$ of periodic functions of $v$ variables, which is based on the mixed differences (see, for instance, \cite{VTmon}, p.31,  in the case of scalar $q$).  
 
\begin{Definition}\label{fcD2}
Let  $\btt =(t_1,\dots,t_v)$ and $\Delta_{\btt}^l f(\bx)$
be the mixed $l$-th difference with
step $t_j$ in the variable $x_j$, that is
$$
\Delta_{\btt}^l f(\bx) :=\Delta_{t_v,v}^l\cdots\Delta_{t_1,1}^l
f(x_1,\dots ,x_v) .
$$
Let $e$ be a subset of natural numbers in $[1,v]$. We denote
$$
\Delta_{\btt}^l (e) :=\prod_{j\in e}\Delta_{t_j,j}^l,\qquad
\Delta_{\btt}^l (\varnothing) := Id \,-\, \text{identity operator}.
$$
We define the class $\bH_{\bq,l}^\ba B$, $l > \|\ba\|_\infty$,   as the set of
$f\in L_\bq(\bbT^v)$ such that for any $e$
\be\label{Bi9}
\bigl\|\Delta_{\btt}^l(e)f(\bx)\bigr\|_\bq\le B
\prod_{j\in e} |t_j |^{a_j} .
\ee
In the case $B=1$ we omit it. It is known (see Theorem \ref{H} below)  that the classes $\bH^\ba_{\bq,l}$ with different $l>\|\ba\|_\infty$ are equivalent. So, for convenience we omit $l$ from the notation. 
\end{Definition}

We now formulate a result, which gives an equivalent description of classes $\bH^\ba_{\bq,l}$. 
 We need some classical trigonometric polynomials. The univariate Fej\'er kernel of order $j - 1$:
$$
\mathcal K_{j} (x) := \sum_{|k|\le j} \bigl(1 - |k|/j\bigr) e^{ikx} 
=\frac{(\sin (jx/2))^2}{j (\sin (x/2))^2}.
$$
The Fej\'er kernel is an even nonnegative trigonometric
polynomial of order $j-1$.  It satisfies the obvious relations
\be\label{FKm}
\| \mathcal K_{j} \|_1 = 1, \qquad \| \mathcal K_{j} \|_{\infty} = j.
\ee
Let $\cK_\bj(\bx):= \prod_{i=1}^v \cK_{j_i}(x_i)$ be the $v$-variate Fej\'er kernels for $\bj = (j_1,\dots,j_d)$ and $\bx=(x_1,\dots,x_v)$.

The univariate de la Vall\'ee Poussin kernels are defined as follows
$$
\cV_m := 2\cK_{2m} - \cK_m.
$$
We also need the following special trigonometric polynomials.
Let $s$ be a nonnegative integer. We define
$$
\mathcal A_0 (x) := 1, \quad \mathcal A_1 (x) := \mathcal V_1 (x) - 1, \quad
\mathcal A_s (x) := \mathcal V_{2^{s-1}} (x) -\mathcal V_{2^{s-2}} (x),
\quad s\ge 2,
$$
where $\mathcal V_m$ are the de la Vall\'ee Poussin kernels defined above.
For $\bs=(s_1,\dots,s_v)\in \bbN^v_0$ define
$$
\cA_\bs(\bx) := \prod_{j=1}^v  \cA_{s_j}(x_j),\qquad \bx=(x_1,\dots,x_v)
$$
and
$$
A_\bs(f) := \cA_\bs \ast f.
$$

The following result is known (see, for instance, \cite{VTmon}, p.32, for the scalar $q$ and \cite{VT32} for the vector $\bq$).

\begin{Theorem}\label{H} Let $f\in \bH^\ba_{\bq,l}$, $\mathbf 1 \le \bq \le \infty$. Then, for $\bs \ge \mathbf 0$
\be\label{H1}
\|A_\bs(f)\|_\bq \le C(\ba,v,l)2^{-(\ba,\bs)}.
\ee
Conversely, from (\ref{H1}) it follows that there exists $B>0$, which does not depend on $f$, such that $f\in \bH^\ba_{\bq,l}B$.
\end{Theorem}

The reader can find results on approximation properties of these classes in the books \cite{VTmon}, \cite{VTbookMA}, and \cite{DTU}.
In this paper we consider the case, when $v=2d$, $d\in \bbN$, $\mathbf{1}\le \bq \le \infty$, and $\ba$ has a special form: $a_j = r$ for $j=1,\dots,2d$. In this case we write $\bW^r_\bq$ and $\bH^{r,2d}_\bq$.  

\section{Some known results on sampling recovery}
\label{kr}

In this section we formulate one known result  in the case $d=1$ -- Theorem \ref{univarR}, which we use later. In addition we formulate some known results in the case $d>1$ in order to demonstrate the difficulty of the problem in this case. 
Functions of the form
$$
t(x) = \sum_{|k|\le n}c_k e^{ikx} =a_0/2+\sum_{k=1}^n
(a_k\cos kx+b_k\sin kx)
$$
are called trigonometric polynomials of order $n$. The set of such polynomials is denoted
by $\cT(n)$.

The  Dirichlet kernel of order $n$
$$
\mathcal D_n (x):= \sum_{|k|\le n}e^{ikx} = e^{-inx} (e^{i(2n+1)x} - 1)
(e^{ix} - 1)^{-1} =\frac{\sin (n + 1/2)x}{\sin (x/2)}
$$
is an even trigonometric polynomial.
Denote
$$
x^j := 2\pi j/(2n+1), \qquad j = 0, 1, ..., 2n.
$$
Clearly,  the points $x^j$,  $j = 1, \dots, 2n$,  are zeros of the Dirichlet
kernel $\mathcal D_n$ on $[0, 2\pi]$.
Therefore,  for any continuous $f$
 $$
I_n(f)(x) := (2n+1)^{-1}\sum_{j=0}^{2n} f(x^j) \mathcal D_n (x - x^j)
$$
interpolates $f$ at points $x^j$: $I_n(f)(x^j)= f(x^j)$, $j=0, 1, ..., 2n$. 

It is easy to check that for any $t\in \cT(n)$ we have $I_n(t)=t$. Using this and the inequality 
$$
\bigl|\mathcal D_n (x)\bigr|\le \min \bigl(2n+1, \pi/|x|\bigr), \qquad
 |x|\le \pi,
$$
we obtain the following Lebesgue inequality
$$
\|f-I_n(f)\|_\infty \le C(\ln(n+1))E_n(f)_\infty,
$$
where $E_n(f)_p$ is the best approximation of $f$ in the $L_p$ norm by polynomials from $\cT(n)$.

The de la Vall\'ee Poussin kernels defined above can also be written as follows
$$
\mathcal V_{n} (x) := n^{-1}\sum_{k=n}^{2n-1} \mathcal D_k (x) = \frac{\cos nx - \cos 2nx}{n (\sin (x/2))^2}.
$$
 
The  de la Vall\'ee Poussin kernels $\mathcal V_{n}$ are even trigonometric
polynomials of order $2n - 1$ with the majorant
$$
\left| \mathcal V_{n} (x) \right| \le C \min \left(n, \ 
 1/(nx^2)\right), \ |x|\le \pi.
$$
 
Consider the following  recovery operator (see \cite{VT23} and \cite{VT51})
$$
R_n (f) := (4n)^{-1}\sum_{j=1}^{4n} f\left(x(j)\right)\mathcal V_n
\left(x - x(j)\right), \qquad x(j) := \frac{\pi j}{2n}.
$$

Let $I_n$ and $R_n$ be the recovery operators defined above and $W^r_q$ stands for $\bW^r_q$ in the univariate case. 

\begin{Theorem}[{\cite{VT51}}]\label{univarR} Let  $1\le q,p \le \infty$ and $r>1/q$. Then
$$
 \varrho_{4m}(W^r_q,L_p) \asymp \sup_{f\in W^r_q}\|f-R_m(f)\|_p \asymp m^{-r+(1/q-1/p)_+}.
$$
In the case  $1<p<\infty$ the above estimates are valid for the operator  $I_m$ instead of the operator  $R_m$. 
\end{Theorem}

We now proceed to the case $d>1$ and give brief comments on the classical Smolyak recovery operators. We refer the reader for a detailed discussion of these and related operators to the books \cite{VTbookNS}, \cite{VTbookMA}, and \cite{DTU}. Let for  $i=1,\dots,d$ operator  $R_n^i$ be the operator  $R_n$ acting with respect to the variable  $x_i$. 
Denote 
$$
 \Delta_s^i := R_{2^s}^i - R_{2^{s-1}}^i,\quad R_{1/2} =0, 
$$
and for  $\bs=(s_1,\dots,s_d)\in \bbN^d_0$
$$
 \Delta_\bs := \prod_{i=1}^d \Delta_{s_i}^i.
$$
Consider the recovery operator (Smolyak operator)
$$
 T_n := \sum_{\bs:\|\bs\|_1\le n} \Delta_\bs.
$$
Operator  $T_n$ uses  $m$ function values with  $m\ll \sum_{k=1}^n 2^k k^{d-1} \ll 2^n n^{d-1}$. 

The following bound was obtained by  S. Smolyak in \cite{Smol}.  
$$
 \sup_{f\in\bW^r_\infty} \|f-T_n\|_\infty \ll 2^{-rn}n^{d-1},\quad r>0.
$$
It was extended to the case $p<\infty$ in  \cite{VT23}:
$$
 \sup_{f\in\bW^r_p} \|f-T_n\|_p \ll 2^{-rn}n^{d-1},\quad r>1/p.
$$
Here are some further results. We have (\cite{VT51})
$$
 \varrho_m(\bW^r_2)_\infty \asymp m^{-r+1/2} (\log m)^{r(d-1)},\quad r>1/2.
$$
The order of optimal recovery in the above case is provided by the Smolyak operator  $T_n$ with an appropriate $n$.
 Also we know (\cite{VT51}) for $q\in (1,\infty)$
$$
  \sup_{f\in \bW^r_q}\|f-T_n(f)\|_\infty \asymp 2^{-(r-1/q)n}n^{(d-1)(1-1/q)}.
$$

Most of the known results on optimal sampling recovery deal with the linear recovery methods. We now give some very brief comments on recent results in this direction and 
refer the reader to the books \cite{DTU}, \cite{VTbookMA} and to the survey paper \cite{KKLT} for a discussion of the previous results in this direction. We are interested in results, which relate the errors of sampling recovery with the Kolmogorov widths for general function classes. We begin with a result from \cite{VT183}. 

\begin{Theorem}[{\cite{VT183}}]\label{VT183T1} There exist two positive absolute constants $b$ and $B$ such that for any   compact subset $\Omega$  of $\bbR^d$, any probability measure $\mu$ on it, and any compact subset $\bF$ of $\cC(\Omega)$ we have
\be\label{kr1}
\ro_{bn}(\bF,L_2(\Omega,\mu)) \le Bd_n(\bF,L_\infty).
\ee
\end{Theorem}

The following generalization of Theorem \ref{VT183T1} to the case $2<p\le \infty$ was obtained in \cite{KPUU2}. 

\begin{Theorem}[{\cite{KPUU2}}]\label{KPUU2} Let $2\le p\le \infty$. There exists a positive absolute constant $C$ such that for any   compact subset $\Omega$  of $\bbR^d$, any probability measure $\mu$ on it, and any compact subset $\bF$ of $\cC(\Omega)$ we have
\be\label{kr2}
\ro_{4n}(\bF,L_p(\Omega,\mu)) \le Cn^{1/2-1/p}d_n(\bF,L_\infty).
\ee
\end{Theorem}

Note that we have $d_n(\bF,L_\infty)$ in the right side of (\ref{kr1}), which is larger than $d_n(\bF,L_2)$. However, it is known that for many function classes we have $d_n(\bF,L_\infty) \asymp d_n(\bF,L_2)$. 
For special sets $\bF$ (in the reproducing kernel Hilbert space setting) the following inequality was proved (see \cite{DKU}, \cite{NSU}, \cite{KU}, and \cite{KU2}):
\be\label{H5a}
\ro_{cn}(\bF,L_2) \le \left(\frac{1}{n}\sum_{k\ge n} d_k (\bF, L_2)^2\right)^{1/2}
\ee
with an absolute constant $c>0$. Here, $d_k (\bF, L_2)$ is the Kolmogorov width of $\bF$ in 
the space $L_2$. 

Known results on the $d_k (\bW^r_q, L_2)$ (see, for instance, \cite{VTbookMA}, p.216): For $1<q\le 2$, $r > 1/q-1/2$ and $2<q<\infty$, $r> 0$
\be\label{H6a}
d_k(\bW^r_q,L_2) \asymp \left(\frac{(\log k)^{d-1})}{k}\right)^{r-(1/q-1/2)_+},\quad (a)_+ := \max(a,0),
\ee
 combined with the (\ref{H5a}) give for $1<q\le \infty$, $r > \max(1/q,1/2)$ the following bounds
$$
 \ro_m(\bW^r_q,L_2) \ll  \left(\frac{(\log m)^{d-1})}{m}\right)^{r-(1/q-1/2)_+}, 
$$
which gives the right orders of decay of the sequences $\ro_m(\bW^r_q,L_2)$ in the case $1<q<\infty$ and $r > \max(1/q,1/2)$ because, obviously, $ \ro_m(\bW^r_q,L_2) \ge  d_m(\bW^r_q,L_2)$.  

The above inequality (\ref{H5a}) can only be used, when the series in its right side converges. 
The above Theorem \ref{VT183T1} is useful even in the case, when the corresponding series diverges.  

\section{Proofs}
\label{Pr}

We discuss separately the upper bounds and the lower bounds. We begin with two results which connect the error of sampling recovery $\ro_m(\bW^K_q,L_p)$ and the error of sparse approximation. 

\subsection{Some connections}
\label{sr}

\begin{Proposition}\label{RNP1} Let $1\le q,p \le \infty$. Assume that for every $\bz\in \Omega^1$ we have $K(\bz,\cdot) \in L_{q'}(\Omega^2)$, $q':=q/(q-1)$. Then we have
 \be\label{RN1}
 \ro_m(\bW^K_q,L_p) \le \sigma_m(K,\cL\cK(p,q'))_{p,q'}.
 \ee
 \end{Proposition}
\begin{proof} Consider an operator $\Psi_m$ of linear recovery
$$
\Psi_m(f,\xi,\bx) := \sum_{j=1}^m f(\xi^j)\psi_j(\bx).
$$
Then we have for $f\in \bW^K_q$
$$
\|f(\bx)-\Psi_m(f,\xi,\bx)\|_p \le \int_{\Omega^2} \left\|K(\cdot,\by)-\sum_{j=1}^m K(\xi^j,\by)\psi_j(\cdot)\right\|_p |\ff(\by)|d\mu_2.
$$
This implies that 
$$
\sup_{f\in\bW^K_q} \|f(\bx)-\Psi_m(f,\xi,\bx)\|_p \le  \left\|K(\bx,\by)-\sum_{j=1}^m K(\xi^j,\by)\psi_j(\bx)\right\|_{p,q'}  .
$$
We now take infimum over sets of points $\{\xi^j\}_{j=1}^m$ and sets of functions $\{\psi_j\}_{j=1}^m$ and complete the proof. 

\end{proof}

For the next simple relation we need a new notation. Define for $p_1,p_2$
$$
\|f(\bx,\by)\|_{L^*_{p_1,p_2}}:=\|f(\bx,\by)\|^*_{p_1,p_2} := \|\|f(\bx,\cdot)\|_{p_2}\|_{p_1},
$$
which means that first we take the norm with respect to $\by$ and after that the norm with respect to $\bx$. 

 \begin{Proposition}\label{RNP2} Let $1\le q \le \infty$. Assume that for every $\bz\in \Omega^1$ we have $K(\bz,\cdot) \in L_{q'}(\Omega^2)$, $q':=q/(q-1)$. Then we have
 \be\label{RN2}
 \ro_m(\bW^K_q,L_\infty) = \sigma_m(K,\cL\cK(\infty,q'))_{L^*_{\infty,q'}}.
 \ee
 \end{Proposition}
\begin{proof} In the same way as in the above proof of Proposition \ref{RNP1} we obtain for $f\in \bW^K_q$
$$
\|f(\bx)-\Psi_m(f,\xi,\bx)\|_\infty 
$$
$$
=\sup_{\bx\in\Omega^1}\left| \int_{\Omega^2} \left(K(\bx,\by)-\sum_{j=1}^m K(\xi^j,\by)\psi_j(\bx)\right) \ff(\by)d\mu_2 \right| =: \sup_{\bx\in\Omega^1}E(\bx,\ff).
$$
Therefore,
$$
\sup_{f\in\bW^K_q}\|f(\bx)-\Psi_m(f,\xi,\bx)\|_\infty= \sup_{f\in\bW^K_q}\sup_{\bx\in\Omega^1}E(\bx,\ff) = \sup_{\bx\in\Omega^1} \sup_{f\in\bW^K_q}E(\bx,\ff) 
$$
$$
=\left\|K(\bx,\by)-\sum_{j=1}^m K(\xi^j,\by)\psi_j(\bx)\right\|^*_{\infty,q'}.
$$
We now take infimum over sets of points $\{\xi^j\}_{j=1}^m$ and sets of functions $\{\psi_j\}_{j=1}^m$ and complete the proof. 
\end{proof}

\subsection{Upper bounds}
\label{ub}

{\bf Proof of upper bounds in Theorem \ref{InT2}.} We begin with the sampling recovery in the $L_2$ norm. We are interested in classes $\bW^K_q$ with $K$ coming from 
a class of functions of two variables. Theorem \ref{VT183T1} gives us the upper bound  
\be\label{N1}
\ro_{bn}(\bF,L_2(\Omega,\mu)) \le Bd_n(\bF,L_\infty),
\ee
which holds for any class $\bF$. So, we need a result on the upper bounds on the Kolmogorov widths.
The following  result is proved in \cite{VT32}.
 
  \begin{Theorem}[{\cite{VT32}}, Theorem 4.2]\label{BiT5} Let $d=1$ and $\bF^\br_1$ denote one of the classes $\bW^\br_1$ or $\bH^\br_1$. Then for $1\le q,p \le \infty$ and $\br > (1,1+\max(1/2,1/q))$   we have 
 $$
\sup_{K\in \bF^\br_1}  d_m(\bW^K_q)_p \asymp m^{-r_1-r_2 + \xi(q,p)} 
 $$
 with $\xi(q,p)$ defined in (\ref{ksi'}).
 \end{Theorem}
We now apply Theorem \ref{BiT5} with  $p=\infty$ and obtain from (\ref{N1}) the following upper bounds 
for the $\ro_{m}(\bW^K_q,L_2)$.

 \begin{Theorem}\label{N1a} Let $d=1$ and $\bF^\br_1$ denote one of the classes $\bW^\br_1$ or $\bH^\br_1$. Then for $1\le q \le \infty$ and $\br > (1,1+\max(1/2,1/q))$   we have 
 $$
\sup_{K\in \bF^\br_1}  \ro_{m}(\bW^K_q,L_2) \ll m^{-r_1-r_2 + \xi(q,\infty)} = m^{-r_1-r_2 + (1/q-1/2)_+}.
 $$
 \end{Theorem}
 
 Note, that the use of inequality (\ref{H5a}) instead of Theorem \ref{VT183T1} will give the same upper bound. 
 
 In the case of measuring error of recovery in the $L_p$ norm with $2<p\le \infty$ we apply Theorem \ref{BiT5} and Theorem \ref{KPUU2} and obtain the following result. 
 
  \begin{Theorem}\label{N2a} Let $d=1$ and $\bF^\br_1$ denote one of the classes $\bW^\br_1$ or $\bH^\br_1$. Then for $1\le q \le \infty$, $2<p\le \infty$, and $\br > (1,1+\max(1/2,1/q))$   we have 
 $$
\sup_{K\in \bF^\br_1}  \ro_{m}(\bW^K_q,L_p) \ll   m^{-r_1-r_2 +(1/2-1/p)+ (1/q-1/2)_+}.
 $$
 \end{Theorem}
 
Theorems \ref{N1a} and \ref{N2a} cover the case $1\le q\le \infty$ and $2\le p\le \infty$. In the case $1\le q\le 2$ and $2\le p\le \infty$, which is addressed in Theorem \ref{InT2}, Theorems \ref{N1a} and \ref{N2a} provide the required upper bounds in Theorem \ref{InT2}. 

{\bf Proof of   Theorem \ref{InT3}.} This proof is similar to the above proof of Theorem \ref{InT2}. We use the following known result on the Kolmogorov widths. 

For $\bq =(q_1,q_2)$, $\bp=(p_1,p_2)$, $1\le q_1 \le p_1 \le \infty$, $1\le q_2,p_2 \le \infty$ denote
 $$
 \br(\bq,\bp) := \begin{cases} (1/q_1-1/p_1,(1/q_2-1/p_2)_+), & 1\le q_1\le p_1 \le 2, \\
 (1/q_1,1/q_2), & 2\le q_1\le p_1\le \infty, p_1>2,\\
 (1/q_1,\max(1/2,1/q_2)), & 1\le q_1<2< p_1 \le \infty.
 \end{cases}
 $$
 
 \begin{Theorem}[{\cite{VT32}}, Theorem 4.1]\label{BiT5a} Let $d=1$ and $\bF^\br_\bq$ denote one of the classes $\bW^\br_\bq$ or $\bH^\br_\bq$. Then for $\bp=(p,\infty)$, $1\le q_1 \le p \le \infty$, $1\le q_2\le \infty$  and $\br > \br(\bq,\bp)$   we have 
 $$
\sup_{K\in \bF^\br_\bq}  d_m(\bW^K_1)_p \asymp m^{-r_1-r_2 + \xi(q_1,p)} 
 $$
 with $\xi(q,p)$ defined in (\ref{ksi'}).
 \end{Theorem}

We now use Theorem \ref{BiT5a} with $p=\infty$, apply Theorems \ref{VT183T1}, \ref{KPUU2} and obtain Theorem \ref{InT3}.

{\bf Proof of   Theorem \ref{InT4}.} We derive it from Theorem \ref{InT3}, which we have proved above. For that we use Proposition \ref{RNP2} with $q=1$. It gives 
 \be\label{RN2a}
 \ro_m(\bW^K_1,L_\infty) = \sigma_m(K,\cL\cK(\infty))_{L^*_{\infty,\infty}} = \sigma_m(K,\cL\cK(\infty))_{L_{\infty,\infty}}.
 \ee
 We now apply Theorem \ref{InT3} with $p=\infty$ and complete the proof.
 
 \subsection{Lower bounds}
 \label{Lb}
 
 {\bf Proof of lower bounds in Theorem \ref{InT2}.} As we pointed out in Section \ref{In} the lower bounds in Theorem \ref{InT2} follow from Theorem \ref{LbT1}. We now give an independent proof. Denote $r:=r_1+r_2$ for $\br=(r_1,r_2)$. It is known (see \cite{VT31}) and it is not difficult to check that $F_r(x-y)$ with $r>2$ belongs to the closure (in the uniform norm) of  $\bW^\br_1$. Therefore, 
 $$
 \sup_{K\in \bF^\br_1}\ro_m(\bW^K_q,L_p) \ge \ro_m(\bW^r_q,L_p)
 $$
 and by Theorem \ref{univarR} we continue
 $$
 \ge m^{-r+(1/q-1/p)_+}.
 $$
 
 {\bf Proof of   Theorem \ref{LbT1}.} In this subsection we need a slightly more general Bernoulli kernels and integral operators related to them (see \cite{VTbookMA}, Section 1.4). In the univariate case, for $a>0$, and $\al \in \bbR$ let
 $$
 F_{a,\al}(x):= 1+2\sum_{k=1}^\infty k^{-a}\cos (kx-\al\pi/2)
 $$
\be\label{Lb1}
  = 1+\sum_{k=1}^\infty k^{-a}(e^{i\al\pi/2}e^{-ikx}+e^{-i\al\pi/2}e^{ikx})
\ee
be the generalised Bernoulli kernel. Clearly, we have $F_{a}(x) = F_{a,a}(x)$, where $F_a(x)$ is defined in (\ref{Bi8}). Define the integral operator, acting on trigonometric polynomials $\phi(x)$, as
\be\label{Lb2}
(I^{(a,\al)}\phi)(x) := (I^{(a,\al)}_x\phi)(x) := ( F_{a,\al} \ast \phi)(x):= \frac{1}{2\pi}\int_{\bbT} F_{a,\al}(x-z) \phi(z)dz.
\ee
The operator $I^{(a,\al)}$ is the multiplier operator:
\be\label{Lb3}
(I^{(a,\al)}\phi)(x) = \hat \phi(0)+ \sum_{k<0}  |k|^{-a}e^{i\al\pi/2} \hat{\phi}(k)e^{ikx}+ \sum_{k>0}k^{-a}e^{-i\al\pi/2} \hat{\phi}(k)e^{ikx}.
\ee
We now define the inverse operator to the operator $I^{(a,\al)}$, acting on the trigonometric polynomials from $\cT(2n)$ (we take $2n$ for convenience in the future use). Define
$$
\cD^{(a,\al)}_{2n}(x) := 1+2\sum_{k=1}^{2n} k^{a}\cos (kx+\al\pi/2)
$$
and the operator (for $g \in \cT(2n)$) 
\be\label{Lb4}
(D^{(a,\al)}g)(x) := (D^{(a,\al)}_xg)(x) := (\cD^{a,\al}_{2n} \ast g)(x):= \frac{1}{2\pi}\int_{\bbT} \cD^{a,\al}_{2n}(x-z) g(z)dz.
\ee
The operator $D^{(a,\al)}$ is the multiplier operator:
\be\label{Lb3a}
(D^{(a,\al)}g)(x) = \hat g(0)+ \sum_{k<0}  |k|^{a}e^{-i\al\pi/2} \hat{g}(k)e^{ikx}+ \sum_{k>0}k^{a}e^{i\al\pi/2} \hat{g}(k)e^{ikx}.
\ee
It is easy to see that for $g \in \cT(2n)$ we have
\be\label{Lb5}
I^{(a,\al)}D^{(a,\al)}g   = g. 
\ee
Also, we have for $g \in \cT(2n)$
\be\label{Lb6}
(D^{(b,\bt)}_yD^{(a,\al)}_x)g(x-y) = (D^{(a+b,\al-\bt)}g)(x-y).
\ee
For a periodic $f \in L_q(\bbT)$ we have for any $y$
\be\label{Lb7}
\|f(\cdot-y)\|_q = \|f\|_q\quad \text{and}\quad \|f(x-y)\|_{(q_1,\infty)}=\|f\|_{q_1}. 
\ee
We now consider the kernel $K(x,y) := \cV_n(x-y)$, where $\cV_n$ is the de la Vall{\'e}e Poussin kernel defined in Section \ref{fc}. Then it is clear that for any $\phi \in \cT(n)$ we have $I_K(\phi) = \phi$. For given $\br = (r_1,r_2)$ and $\bq=(q_1,\infty)$ we normalize the kernel $K(x,y)$ in such a way that it belongs to $ \bW^\br_\bq$. We have $\cV_n \in \cT(2n)$ and, therefore, we can apply the above relations (\ref{Lb5})--(\ref{Lb7}). Define  
$$
\ff(x,y) := (D^{(r_2,r_2)}_yD^{(r_1,r_1)}_x)K(x,y) 
$$
and use (\ref{Lb6})
$$
 = (D^{(r_1+r_2,r_1-r_2)}\cV_n)(x-y).
$$
Using (\ref{Lb7}), we obtain
$$
\|\ff\|_{(q_1,\infty)} = \|D^{(r_1+r_2,r_1-r_2)}\cV_n\|_{q_1}.
$$
We now use the Bernstein inequality (see \cite{VTbookMA}, p.20) and the known bound for the $\|\cV_n\|_{q_1}$ (see \cite{VTbookMA}, p.9) and continue
$$
\le Cn^{r_1+r_2 +1-1/q_1}.
$$
Thus, there exists a positive absolute constant $c$ such that 
\be\label{Lb8}
\|cn^{-(r_1+r_2 +1-1/q_1)}\ff\|_{(q_1,\infty)} \le 1. 
\ee
By the property (\ref{Lb5}) we obtain
\be\label{Lb9}
cn^{-(r_1+r_2 +1-1/q_1)}K \in \bW^\br_{(q_1,\infty)}.
\ee
In order to obtain some lower estimates in our Problem K setting we need known results on sampling recovery. We formulate these results in the whole generality despite of the fact that we only need a special case of it.  
Denote for $\bN = (N_1,\dots,N_d)$, $N_j\in \bbN_0$, $j=1,\dots,d$,
 $$
\Pi(\bN,d) := \{\bk \in \bbZ^d\,:\, |k_j|\le N_j, j=1,\dots,d\}
$$
and
$$
\cT(\bN,d) := \left\{f=\sum_{\bk\in \Pi(\bN,d)} c_\bk e^{i(\bk,\bx)}\right\},\quad \vartheta(\bN) :=\prod_{j=1}^d (2N_j+1). 
$$
In this section $\Omega = \bbT^d$ and $\mu$ is the normalized Lebesgue measure on $\bbT^d$. The following Lemma \ref{qpL1} was proved in \cite{VT203}.

\begin{Lemma}[{\cite{VT203}}]\label{qpL1} Let $1\le q\le p\le \infty$ and let  $\cT(\bN,d)_q$ denote the unit $L_q$-ball of the subspace $\cT(\bN,d)$. Then we have for $m\le \vartheta(\bN)/2$ that
$$
\varrho_m^o(\cT(2\bN,d)_q,L_p) \ge c(d)\vartheta(\bN)^{1/q-1/p}  .
$$
\end{Lemma}

Combining Lemma \ref{qpL1} with the above arguments we obtain Theorem \ref{LbT1}.

\section{Discussion}
\label{D}

In this section we present one more (probably, it is the only additional one) of known results on relations between the Problem $K$ and the sparse approximation with adaptive dictionaries. Namely, we discuss here some known results on connections between 
numerical integration and nonlinear approximation. The reader can find these and other related results in \cite{VTbookMA}, Section 6.3 and in \cite{VT170}. We begin with the necessary definitions. We formulate the numerical integration problem in a general setting.  Numerical integration seeks good ways of approximating an integral
$$
\int_\Omega f(\bx)d\mu
$$
by an expression of the form
\be\label{A.1}
\La_m(f,\xi) :=\sum_{j=1}^m\la_jf(\xi^j),\quad \xi=(\xi^1,\dots,\xi^m),\quad \xi^j \in \Omega,\quad j=1,\dots,m. 
\ee
It is clear that we must assume that $f$ is integrable and defined at the points
 $\xi^1,\dots,\xi^m$. Expression (\ref{A.1}) is called a {\it cubature formula} $(\xi,\La)$ (if $\Omega \subset \bbR^d$, $d\ge 2$) or a {\it quadrature formula} $(\xi,\La)$ (if $\Omega \subset \bbR$) with knots $\xi =(\xi^1,\dots,\xi^m)$ and weights $\La:=(\la_1,\dots,\la_m)$. 
 
 For a function class $\bW$ we introduce a concept of error of the cubature formula $\La_m(\cdot,\xi)$ by
\be\label{A.3}
\La_m(\bW,\xi):= \sup_{f\in \bW} |\int_\Omega fd\mu -\La_m(f,\xi)|. 
\ee
The problem of finding optimal in the sense of order cubature formulas for a given class is of special importance. This means 
that we are looking for a cubature formula $\La_m^{opt}(\bW,\xi)$ such that
\be\label{A.4}
\La_m^{opt}(\bW,\xi) \asymp \inf_{\xi,\La}\La_m(\bW,\xi)=: \kappa_m(\bW).
\ee

We now present a setting of this problem for the $\bW^K_q$ classes.  Let $1\le q\le \infty$. We define a set $\mathcal K_q$ of kernels possessing the following properties. 
 Let $K(\bx,\by)$ be a measurable function on $\Omega^1\times\Omega^2$.
 We assume that for any $\bx\in\Omega^1$ we have $K(\bx,\cdot)\in L_q(\Omega^2)$; for any $\by\in \Omega^2$ the $K(\cdot,\by)$ is integrable over $\Omega^1$ and $\int_{\Omega^1} K(\bx,\cdot)d\bx \in L_q(\Omega^2)$. As above, for $1\le q\le \infty$ and a kernel $K\in \mathcal K_{q'}$, $q':=q/(q-1)$, we define the class
\be\label{H.5}
\bW^K_q :=\left\{f:f=\int_{\Omega^2}K(\bx,\by)\varphi(\by)d\by,\quad\|\varphi\|_{L_q(\Omega^2)}\le 1\right\}.  
\ee
Then each $f\in \bW^K_q$ is integrable on $\Omega^1$ (by Fubini's theorem) and defined at each point of $\Omega^1$. We denote for convenience
$$
 J_K(\by):=\int_{\Omega^1}K(\bx,\by)d\bx.
$$

For a cubature formula $\Lambda_m(\cdot,\xi)$ we have
$$
\Lambda_m(\bW^K_q,\xi) = \sup_{\|\varphi\|_{L_q(\Omega^2)}\le 1} |\int_{\Omega^2}\bigl( J_K(\by)-\sum_{\mu=1}^m\lambda_\mu K(\xi^\mu,\by)\bigr)\varphi(\by)d\by|=
$$
\be\label{H.6}
=\left\| J_K(\cdot)-\sum_{\mu=1}^m\lambda_\mu K(\xi^\mu,\cdot)\right\|_{L_{q'}(\Omega^2)}.
\ee
 
Thus, we obtain the following relation.

 \begin{Proposition}\label{RNP3} Let $1\le q \le \infty$. Assume that $K\in \mathcal K_{q'}$, $q':=q/(q-1)$. Then we have
 \be\label{RN3}
 \kappa_m(\bW^K_q) := \inf_{\la_1,\dots,\la_m;\xi^1,\dots,\xi^m}\La_m(\bW^K_q,\xi) = \sigma_m(J_K,\cD(K)_{L_{q'}},
 \ee
 where $\cD(K) := \{K(\bz,\by)\}_{\bz\in \Omega^1}$.  
 \end{Proposition}
 
 In Proposition \ref{RNP3} we have the sparse approximation problem with an adaptive dictionary. Namely, we approximate $J_K$ with respect to $\cD(K)$. 

%{\bf Comment.} Our results from Section \ref{KB} show that the problem of the Kolmogorov widths of classes $\bW^K_q$ is closely related to the problem of 
%best sparse approximation of the kernel $K(\bx,\by)$ with respect to the bilinear system 
%$\Pi(d)$. The system $\Pi(d)$ does not depend on the kernel $K(\bx,\by)$. Our results of this section show that in the case of linear sampling recovery the optimal errors for classes $\bW^K_q$ are connected with the problem of 
%best sparse approximation of the kernel $K(\bx,\by)$ with respect to the system 
%$\Pi(K,p)$, which depends on the kernel. In the case of numerical integration we need to study sparse approximations of a special function $J_K$ obtained from the kernel $K(\bx,\by)$ with respect to the system $\{K(\bu,\by)\}_{\bu\in \Omega^1}$ determined by the kernel. 

  {\bf Acknowledgements.}  
This work was supported by the Russian Science Foundation under grant no. 23-71-30001, https://rscf.ru/project/23-71-30001/, and performed at Lomonosov Moscow State University.

 \Addresses
 
\end{document}